\documentclass[conference]{IEEEtran}
\usepackage{fullpage}
\usepackage{amssymb, bbold}
\usepackage{amsmath}
\usepackage{mdwmath}
\usepackage{mdwtab}
\usepackage{mathenv}
\usepackage{cite}

\newcommand{\src}{\mathcal{S}}
\newcommand{\snk}{\mathcal{T}}

\newcommand{\F}{\mathbb{F}}
\newcommand{\rk}{\mathsf{rank}}
\newcommand{\p}{\boxtimes} 
\newcommand{\inst}{node-capacitated multicommodity instance }
\newcommand{\insts}{node-capacitated multicommodity instances }

\newcommand{\supp}{\mathsf{supp}}
\newcommand{\krn}{\otimes}
\newcommand{\ind}[1]{\mathbb{1}_{#1}}
\newcommand{\cmat}{L}
\newcommand{\cols}{\mathcal{M}}
\newcommand{\csub}{D}

\newtheorem{thm}{Theorem}[section]

\newtheorem{lemma}[thm]{Lemma}
\newtheorem{cor}[thm]{Corollary}

\newtheorem{obs}[thm]{Observation}

\newtheorem{definition}[thm]{Definition}

\begin{document}

\title{Multicut Lower Bounds via Network Coding}

\author{\IEEEauthorblockN{Anna Blasiak}
\IEEEauthorblockA{Cornell University}}

\maketitle

\begin{abstract}
We introduce a new technique to certify lower bounds on the multicut size using network coding.  In directed networks the network coding rate is not a lower bound on the multicut, but we identify a class of networks on which the rate is equal to the size of the minimum multicut and show this class is closed under the strong graph product.  We then show that the famous construction of Saks et al. that gives a $\Theta(k)$ gap between the multicut and the multicommodity flow rate is contained in this class.  This allows us to apply our result to strengthen their multicut lower bound, determine the exact value of the minimum multicut, and give an optimal network coding solution with rate matching the multicut. 
\end{abstract}
\section{Introduction}

The multicut problem is a fundamental graph partitioning problem in which we are given a network with $k$ source-sink pairs and we are asked to find the minimum size edge set that, when removed, disconnects all source-sink pairs.  It has applications in network robustness where we may want guarantees that the multicut is large implying our network will still be connected even after the failure of many edges.  Alternatively, we may want to compute a small multicut in order to determine an efficient way to stop the spread of a contagion in a network.  

The problem is known to be NP-hard to compute and even NP-hard to approximate \cite{dahlhaus1994complexity, chuzhoy}, and so the focus of previous work has been on approximation algorithms.  
As is the case with many graph problems, the directed version seems far more difficult than the undirected one.  The best approximation algorithm for undirected graphs is $O(\min(\log n, \log k))$  \cite{LeightonRao, garg-approx}.  But for directed graphs, the best approximation is $\tilde{O}(n^{11/23})$ \cite{cut-approx}.


All of the approximation algorithms \cite{gupta2003improved, cut-approx, cheriyan2001approximating} to date use the natural linear programming relaxation.  The dual of the linear program is the maximum multicommodity flow problem, which looks to maximize the total flow that can be sent between the source-sink pairs.  This technique is limited by the integrality gap of this linear program, also called the flow-cut gap, 
and this gap is known to be large.  
In undirected graphs the flow-cut gap is equal to the best known approximation ratio, $\Theta(\min(\log k, \log n))$ \cite{LeightonRao, garg-approx}.  For directed graphs, the paper of Saks et al. \cite{Saks} shows that the trivial upper bound of $k$ on the flow-cut gap is tight up to constant factors.  Recently, Chuzhoy et al. showed that the gap is large when parameterized by $n$ as well, and is $\tilde{\Omega}(n^{1/7})$ \cite{chuzhoy}.  Thus, the lower bound given by the maximum multicommodity flow problem isn't strong enough to allow for improved approximation algorithms in the undirected case or in the directed case parameterized by $k$.

The focus of this work is to consider the possibility of a stronger lower bound via the rate of the network coding version of the maximum multicommodity flow problem.  As a generalization of flow, network coding's rate is at least the flow rate and could perhaps yield better lower bounds.  Such a technique has been extremely successful in the multicast problem;  the coding rate there is known to be easy to compute and equal to the cut bound, unlike the flow solution which can be much smaller \cite{li2003linear, Jaggi, agarwal2004advantage}.   

The obstacles to a lower bound via network coding for the multicut problem differ between undirected and directed settings.  In undirected graphs, the coding rate is a lower bound \cite{Harvey2006capacity}, but it is unlikely to lead to improved lower bounds as the network coding rate is conjectured to be equal to the flow rate \cite{li2004network}.  In directed graphs, the coding rate can be a factor $k$ larger than the flow rate \cite{li2004network, harvey2004comparing}, so it would have potential to give a tight lower bound.  However, here the network coding rate is not a lower bound on the multicut and can even be a factor $k$ larger than the cut \cite{ACLY}.

Though the multicut isn't an upper bound on the coding rate, there exist many related cut upper bounds.  An easy entropy argument shows that a cut that disconnects all sinks from all sources is an upper bound on the network coding rate, and there has been work devoted to expanding that idea with more complicated entropy arguments \cite{harvey2004comparing, harvey2005tighter, KramerSavari, song2003zero}.  But, to our knowledge, there has been no prior work investigating conditions under which the network coding rate is a lower bound on the multicut in directed graphs, and as such, it is the primary focus of our contribution. 

\subsection{Our Contributions}
In this paper we introduce a new technique to certify lower bounds on the multicut size using network coding.  
We identify a property of a linear network code that guarantees the code is a lower bound on the multicut.   We also show that for the strong graph product of any two networks with such codes, this property is preserved.  The following theorem describes one consequence of our main result:

\begin{thm}
Given a network $G$ in which the optimal multicommodity flow solution consists of a set of node-disjoint paths, there is a product operation in which the optimal network coding rate is equal to the minimum multicut in the $k$-fold product of $G$.  
\label{thm:informal}
\end{thm}

By applying this theorem to a directed path of length $n$ with source and sink at the ends, we give a new lower bound on the multicut in the construction of Saks et al.  Our proof strengthens Saks's result and provides a tight lower bound on the multicut (see Corollary \ref{cor:saks}).  Further, it constructs an elegant network coding solution for the construction that has rate equal to the multicut and a $k - o(k)$ factor larger than the multicommodity flow rate.




\section{Preliminaries}

We begin by defining the class of networks for which we analyze the multicut and network coding rates.  The definition is tailor-made for graph products.

\begin{definition}
A \emph{\inst} is given by a tuple $N = (G, \src,\snk, f)$ where $G=(V,E)$ is an undirected graph, $\src$ and $\snk$ are an ordered list of sources and sinks (separate from $G$) such that the $i^{th}$ source and sink are paired, and $f: \src \cup \snk \mapsto 2^V $ is a function that maps each source and sink to a subset of nodes.  The instance network can be formed by adding nodes for each element in $\src$ and $\snk$ to $G$ and adding directed edges $(s,v)$ for all $s \in \src, v \in f(s)$ and $(u,t)$ for all $t \in \snk, u \in f(t)$.  We reserve $n$ to denote $|V|$.\label{def:inst}
\end{definition}

It is easier for us to work with node-capacitated networks, but any node-capacitated network can be transformed into an equivalent edge-capacitated network by replacing each node with two nodes with a single directed edge between them.  For this reason, even though the graph $G$ is undirected, the network we are considering is far from undirected.

We will show that under certain conditions linear network codes and multicuts in these network instances can be composed under the following product operation.

\begin{definition}
The strong product of two instances $N_1 = (G_1, \src_1, \snk_1,f_1)$ and $N_2 = (G_2, \src_2, \snk_2, f_2)$ is the instance $N_1 \p N_2 = (G_1 \p G_2, \src, \snk, f)$ where $G_1 \p G_2$ is the strong graph product of $G_1 = (V_1, E_1)$ and $G_2 = (V_2, E_2)$: 
$$\begin{aligned}
V(G_1\p G_2) &= V_1 \times V_2\\ 
E(G_1\p G_2) & = \{ ((u,v), (u',v')) | (u,v) \ne (u',v')\\
 & \phantom{space} u=u' \text{or } (u,u') \in E_1, \\
 & \phantom{space} v=v' \text{or } (v,v')\in E_2 \}.
\end{aligned}$$
The set of sources $\src = \src_1 \cup \src_2$.
 The function $f$ is defined by 
$$f(s) = 
\begin{cases}
f_1(s) \times V_2 & \mbox{if $s \in \src_1$} \\
V_1 \times f_2(s) & \mbox{if $s \in \src_2$} 
\end{cases}$$
The sinks $\snk$ and function $f(\snk)$ are defined in the corresponding manner.   
\end{definition} 

Our analysis relies heavily on matrices and we now define the notation and important definitions.  Let $A[i,j]$ denote the $(i,j)^{th}$ entry of $A$, $A[i, -]$ the $i^{th}$ row, and $A[-,j]$ the $j^{th}$ column.  Correspondingly, for a vector $v$, let $a[i]$ denote the $i^{th}$ entry of $a$.  

\begin{definition}
The Kronecker product of a $p \times q$ matrix $A$ and $p' \times q'$ matrix $B$ is a $pp' \times qq'$ matrix
$$A \krn B = \left[
\begin{matrix} 
a[1,1] B & \cdots & a[1,q]B \\ 
\vdots & \ddots & \vdots \\ 
a[p,1]B & \cdots & a[p,q]B 
\end{matrix} \right].$$
\end{definition} 

\begin{definition}
The \emph{support} of a vector $v \in \F^{|A|}$, denoted $\supp(v)$, with entries indexed by the set $A$ is the subset $A' \subseteq A$ such that $v[a] \ne 0$ iff $a \in A'$.  In other words, $\supp(v)$ is the support of the function $f: A \mapsto \F$ such that $f(a) = v[a]$.
\end{definition}

We will overload functions defined on elements of sets to also be defined on subsets.  For a function $f:2^A \mapsto 2^B$ and a subset $A' \subseteq A$, we define $f(A') := \bigcup_{a \in A'} f(a)$.   For a function $f: 2^A \mapsto \mathbb{R}$,  we define $f(A') := \sum_{a \in A'} f(a)$.  Often we will use the additional shorthand of denoting $f(A)$ by $f$. 

\section{Codes and Cuts}

There are some subtleties to defining network coding solutions in graphs with cycles \cite{ARL-thesis}.  To avoid these issues we restrict our definition of a network code to  include an ordering on nodes that specifies possible dependencies between message vectors.  

\begin{definition}
A \emph{linear network code} $(\F, r, \pi, \cmat)$ of a \inst $((V,E),\src, \snk,f)$ specifies a finite field $\F$, a function $r(s): \src \mapsto \mathbb{N}$, an ordering $\pi:V \mapsto [n]$ on nodes in $V$, and a $n \times r(\src)$ coding matrix $\cmat$.  The rows of $\cmat$ are labeled with vertices $V$ and the columns by messages $\cols := \bigcup_{s \in \src} \cols(s)$, where $\cols(s) := \{(s,1),\ldots, (s,r(s))\}$.  Defining $N(v)$ to be  $\{v\} \cup \{u \in V | \pi(u) < \pi(v), (u,v) \in E\}$, we have that: 

For $v \in V$, $\exists a_v \in \F^{1\times n}$ such that 
\begin{enumerate}
\item $\{v\} \subseteq \supp(a_v) \subseteq N(v),$
\item $\supp(a_v \cmat) \subseteq \cols(f^{-1}(v)).$
\end{enumerate}

\label{def:netcod}
\end{definition}

The $v^{th}$ row of the matrix $\cmat$ describes the linear combination over $\F$ of messages that are sent by node $v$ to all its neighbors in the code.  The existence of vector $a_v$ guarantees that $v$ can compute this linear combination using the messages of adjacent nodes that come earlier in the ordering $\pi$.  In particular, node $v$ can determine its message using $\frac{1}{a_v[v]}\sum_{v' \in N(v)\setminus \{v\}} a_{v}[v']\cmat[v',-]$ and the information from the sources entering node $v$.  

\begin{definition}
A linear network code $(\F, r, \pi, \cmat)$ of a \inst $((V,E),\src, \snk,f)$ is \emph{decodable} with rate $p$ if there is a subset $\csub$ of messages $\cols$ of $\cmat$ with $|\cols| - |\csub| = p$ such that:

For each message $m = (s_i,j) \in \cols \setminus \csub$, $\exists d_{m} \in \F^{1 \times n}$ such that \begin{enumerate}
\item $\supp(d_{m})\subseteq f(t_i)$
\item $\{m\} \subseteq \supp(d_{m} \cmat) \subseteq \{m\} \cup \csub.$  
\end{enumerate}
\label{def:decode}
\end{definition}

Definition \ref{def:decode} guarantees that for for a message $m \in \cols(s_i)$, the sink $t_i$ can decode $m$ assuming that the messages in $\csub$ are fixed and known to all the receivers.  The idea that we can set some messages as fixed is an unusual, but natural generalization of the standard way to describe a linear code.  It will allow us to write the coding matrices in a much nicer form. 

\begin{obs}
\label{obs:dis1}
A network code that sends source messages along $p$ node-disjoint paths is a linear network code that is decodable with rate $p$.
\end{obs}
\begin{proof}
The matrix $\cmat$ has a column for each path that is an indicator vector for the path, and the set $D = \emptyset$.
\end{proof}


\begin{definition}
A \emph{multicut} of a \inst $N = ((V,E), \src, \snk, f)$ is a subset of nodes $M \subseteq V$ such that removing the vertices of $M$ from $N$ disconnects all paths between all $s_i - t_i$ pairs.  
\end{definition}

It will be convenient for us to represent subsets of the vertices of a network in terms of an \emph{indicator matrix}.  For a subset $A \subseteq V$, the matrix $I_A$ will be a $n \times |A|$ matrix with rows indexed by nodes $v \in V$  and columns indexed by nodes $w \in A$ where entry $[v,w] = 1$ if $v = w$ and zero otherwise.  

\begin{definition}
We call a linear network code $C = (\F, r, \pi, \cmat)$ of a \inst $N$ \emph{$\rho$-certifiable} if 

\begin{enumerate} 
\item There are cliques $K(v) \subseteq N(v), \; \forall \; v \in V$ such that $C$ continues to satisfy all of the properties prescribed in the definition of a linear network code (Definition~\ref{def:netcod}) if we replace all occurrences of $N(v)$ in that definition with $K(v)$ for all $v \in V$.
\item For any multicut $M$ of $N$, $\rk(\cmat^T I_M) \ge \rho$. 
\end{enumerate}

\label{def:cert}
\end{definition}

The certifiable property implies that $\rho$ is a lower bound on the size of the multicut: $|M| = \rk(I_M) \ge \rk(\cmat^T I_M) \ge \rho$.  The restriction on the coding matrix given by property 1 will allow us to compose together certifiable coding matricies to get a coding matrix that is certifiable for $N_1 \p N_2$ as well.  Notice that we don't need the matrix to be decodable with any rate for it to be certifiable.

\begin{obs}
Any coding solution consisting of $r$ disjoint paths is $r$-certifiable.  
\label{obs:dis2}
\end{obs}

\begin{proof}
Let $(\F,r,\pi,\cmat)$ be the linear code describing the disjoint path solution.

Observe that $(\cmat^TI_M)[i,j] \ne 0$ iff path $i$ intersects node $j$ of $M$.  $M$ is a multicut, so no row $(\cmat^TI_M)[i,-]$ can be the zero vector.  Further, the paths are disjoint, so each column $(\cmat^TI_M)[-,j]$ can have at most one non-zero entry.  Thus, $\rk(\cmat^TI_M) = r,$ the number of rows in $\cmat^T$.  Further, if $v$ belongs to a disjoint path $P$ then $v$ can compute its message using only its predecessor in $P$, thus Definition \ref{def:netcod} will still hold if we use the subset of $N(v)$ consisting of $v$ and its predecessor in $P$, a 2-clique.
\end{proof}


\section{Preserving Properties in Products}

Our main theorem shows how to combine linear network codes in two networks to obtain a linear network code in their product, preserving both decodability and certifiability.

\begin{thm}
Let $$N_1=(G_1=(V_1,E_1), \src_1, \snk_1, f_1) \text{   and}$$ $$N_2=(G_2=(V_2,E_2), \src_2, \snk_2, f_2)$$ be \insts with linear coding solutions $C_1= (\F, r_1, \pi_1, \cmat_1)$ and $C_2 = (\F, r_2, \pi_2, \cmat_2)$.

There is a linear network coding solution $C$ for $N_1 \p N_2$ with coding matrix $\left[I_{n_1} \krn \cmat_2, \cmat_1 \krn I_{n_2} \right ]$ such that:

\begin{enumerate}

\item If $C_1$ and $C_2$ are decodable with rates $p_1, p_2$ respectively then and $C$ is decodable with rate $p := n_1p_2+n_2p_1 - p_1|f_2(\snk_2)| $. 

\item If $C_1$ and $C_2$ are $\rho_1$ and $\rho_2$ certifiable respectively, then $C$ is $\rho$-certifiable, $\rho := \left(n_1\rho_2+n_2\rho_1 - \rho_1|f_2(\src_2)|\right)$, for $N_1 \p N_2$.

\end{enumerate}

\label{thm:mainresult}
\end{thm}

Before proving the main theorem we show how it applies to give an improvement to the Saks et al. construction.  The network in the construction of Saks et al. is the $k$-fold strong product of the network $\mathcal{P}_n = (P_n,\src = \{s\}, \snk = \{t\},f)$ where $P_n = p_1p_2\ldots p_n$ is a path of length $n$ and $f(s) = p_1, f(t) = p_n$.  Let $\mathcal{P}_n^{\p k}$ denote the Saks et al. graph parameterized by $k$ and $n$.

\begin{cor}
The size of the minimum multicut and the rate of the optimal network coding solution of $\mathcal{P}_n^{\p k}$ is $n^k-(n-1)^k$. 
\label{cor:saks}
\end{cor}

This bound on the multicut is tight and an improvement over the lower bound of $k(n-1)^{k-1}$ given in Saks et al.\cite{Saks}.



\begin{proof}
From Observations \ref{obs:dis1} and \ref{obs:dis2} we know that $\mathcal{P}_n$ has a linear network code $$C = (\F_2, r: r(s) = 1, \pi: \pi(p_i) = i, \mathbf{1}_n)$$ that is decodable with rate $1$ and $1$-certifiable.

We will fix $n$ and apply Theorem \ref{thm:mainresult} inductively on $k$ to show that there is a code $C_k$ for $\mathcal{P}_n^{\p k}$ is $\rho_k$-certifiable and decodable with rate $p_k$, where  $\rho_k = p_k = n^k - (n-1)^k$.  The preceding paragraph establishes that $C_1 = C$ satisfies the base case.  Now, assuming true for $k$, we show for $k+1$:

We apply Theorem \ref{thm:mainresult} to $N_1 = \mathcal{P}_n$ and $N_2 = \mathcal{P}_n^{\p k}$.  By our inductive hypotheis, we have codes $C_1$ and $C_2$ with the required conditions, and now the theorem implies that 
$$\begin{aligned}
\rho_{k+1} & = \rho_k n + \rho n^k - \rho_k |f(\src)| \nonumber \\
& = \rho_k(n-1) + n^k \;\;\; \text{by } \rho = 1, |f(\src)| = 1 \nonumber \\
& = n^{k+1} - (n-1)^{k+1} 
\end{aligned}
$$
The same proof applies to the coding rate because $p = 1, |f(\snk)| = 1$.

Further, note that $|f(\snk_{\mathcal{P}_n^{\p k}})| = n^k-(n-1)^k$ as well, because For $A \subset V_1$ and $B \subset V_2$, the set $A \times B$ has cardinality $|A|n_2 + |B|n_1 - |A||B|$, and again the same inductive proof holds because $|f(\snk)|=1$.  This gives us that $f(\snk_{\mathcal{P}_n^{\p k}})$ is an optimal multicut.  Additionally, $f(\snk_{\mathcal{P}_n^{\p k}})$ cuts all sources from all sinks and therefore gives a tight upper bound on the coding rate. 
\end{proof}

The same proof also implies the following more general corollary, giving us a large set of graphs where the coding rate is a lower bound on the multicut and better than the flow bound.

\begin{cor}
If a \inst $N = (G, \src, \snk, f)$ has a flow solution consisting of $r$ disjoint paths, and $|f(\src)|=|f(\snk)| = r$, then $N^{\p k}$ has an optimal coding rate equal to the size of the optimal multicut equal to $n^k - (n-r)^k$. 
\end{cor}

The proof of the Theorems mostly falls out of manipulation of the Kronecker product, in particular, we repeatedly use of the mixed-product property which states that $(A \krn B)(C \krn D) = AC \krn BD$ if the dimensions match correctly.

To aid in the proof of the second part of Theorem \ref{thm:mainresult}, we begin with some definitions and lemmas whose proofs will come later.

\begin{definition}
A \emph{lower block triangular matrix} is a block matrix such that the blocks above the main diagonal blocks are identically zero.
\end{definition}

\begin{lemma}
If the main diagonal blocks of a lower block triangular matrix have ranks $r_1, r_2, \ldots, r_l$ respectively, then the lower block triangular matrix has rank at least $\sum_{i=1}^l r_i$.
\label{lem:lowertri}
\end{lemma}

The following Lemma is the generalization of a critical Lemma from the Saks et al. proof.  

\begin{lemma}
For every multicut $M$ of $N_1 \p N_2$ and every vertex $u \in V_1$ there is a multicut $M_u$ of $N_2$ such that $K_1(u) \times  M_u \subseteq M$.
\label{lem:extracut} 
\end{lemma} 
Note that by the symmetry of the product operation, Lemma \ref{lem:extracut} also implies that the result holds when we switch the roles of $N_1$ and $N_2$.

Now we come to proving our main theorem.  To avoid confusion, we will reserve $u$ to denote nodes in $V_1$ and $v$ for $V_2$.

\begin{proof}[Proof of Theorem \ref{thm:mainresult}]
We define a linear network code $C = (\F,r,\pi, \cmat)$ on $N_1 \p N_2 = (G_1 \p G_2, \src, \snk, f)$.  It has
$$ r(s) = 
\begin{cases}
r_1(s)n_2  & \mbox{if $s \in \src_1$} \\
r_2(s)n_1  & \mbox{if $s \in \src_2$} 
\end{cases}
$$  
The ordering $\pi$ will be given by $\pi((u,v)) = n_2(\pi_1(u)-1)+\pi_2(v)$, which corresponds to a lexicographic ordering of $(\pi_1(u), \pi_2(v))$, and 
$\cmat = \left[I_{n_1} \krn \cmat_2, \cmat_1 \krn I_{n_2}\right].$

In $\cmat$, the rows are labeled by vertices $(u,v) \in V_1 \times V_2$ and the columns are labeled with messages $\cols = (\cols_1 \times V_2) \cup (V_1 \times \cols_2)$. 

\noindent \textbf{$\mathbf{C}$ is a linear network code for $\mathbf{N_1 \p N_2}$}

We show that $C$ satisfies Definition \ref{def:netcod}.  Let $a_u$ and $a_v$ be the vectors that satisfy Definition~\ref{def:netcod} for $u \in V_1$ and $v \in V_2$ for $C_1$ and $C_2$ respectively.  Now, set $a = a_u \krn a_v$.  We claim that $a$ satisfies Definition~\ref{def:netcod} for $(u,v) \in V_1 \times V_2$ for $N_1 \p N_2$.   

First, note that $\supp(a) = \supp(a_u) \times \supp(a_v)$, giving us that $\supp(a) \subseteq N(u) \times N(v) \subseteq N((u,v))$ as wanted.  Additionally, $a_u[u] \ne 0, a_v[v] \ne 0$ implies that $a[(u,v)] \ne 0$, and $\{(u,v)\} \subseteq \supp(a)$.  

The fact that $\supp(a\cmat) \subseteq \cols(f^{-1}((u,v)))$ follows from the mixed-product property:
$$\begin{aligned}
\supp(a\cmat) & = 
\supp( a_u \krn a_vL_2) \cup \supp( a_uL_1 \krn a_v) \\
& \subseteq (V_1 \times \cols_2(f_2^{-1}(v))) \cup (\cols_1(f_1^{-1}(u)) \times V_2) \\ 
& = \cols(f^{-1}((u,v)))
\end{aligned}$$

\noindent  \textbf{$\mathbf{C}$ is decodable with rate} $\mathbf{p}$

Let $D_1 \subset \cols_1$, $D_2 \subset \cols_2$, and $d^1_{c}, d^2_{c'}$ for $c \in \cols_1 \setminus D_1, c' \in \cols_2 \setminus D_2$ be the subsets and vectors showing that $C_1$ and $C_2$ satisfy Definition~\ref{def:decode}.

We will show that $C$ is $p$-decodable with $$D = (D_1 \times V_2) \cup (V_1 \times D_2) \cup (\cols_1 \times f_2(\snk_2)).$$  Note that 
$|D| = |D_1|n_2 + |D_2|n_1 + (p_1 - |D_1|) |f_2(\snk_2)|,$ and thus $
|\cols| - |D| = p$ as needed.  

We first consider message $m = (u,m_2) = (u,(s'_i,j)) \in (V_1 \times \cols_2)\setminus D.$  Let $d_{m} = \ind{u} \krn d^2_{m_2}.$ We have that $\supp(d_{m}) \subseteq \{u\} \times f_2(t'_i) \subseteq f(t'_i)$.

Additionally,  
$$\begin{aligned}
\supp(d_{m}\cmat) & = \supp\left(\left[ \ind{u} \krn d^2_{m_2}\cmat_2, \ind{u}\cmat_1 \krn d^2_{m_2}\right]\right)\\
& \subseteq \left(\{u\} \times \left( \{m_2\} \cup D_2\right) \right) \cup \left( \cols_1 \times f_2(t_i)\right) \\
& \subseteq \{m\} \cup D 
\end{aligned}$$

Finally, because $\{m_2\} \subseteq \supp(d^2_{m_2}\cmat_2)$, we also have $\{m\} \subseteq \supp(d_m\cmat)$, as needed.

Now we consider message $m = (m_1, v) = ((s_i,j), v) \in \cols_1 \times V_2.$  Similar to the previous case, we define $d_{m} = d^1_{m_1} \krn \ind{v}$ and by parallel arguments, we have that $\supp(d_m) \subseteq f(t_i)$ and $\{m\} \subseteq \supp(d_m\cmat)$.  

To determine the set that contains the support of $d_m\cmat$ we can write down the same set as before, but because $D$ is not symmetric, we can't come to our desired conclusion. 
$$
\begin{aligned}
\supp(d_{m}\cmat) & = \supp \left(\left[ d^1_{m_1} \krn \ind{v}\cmat_2, d^1_{m_1}\cmat_1 \krn \ind{v}\right] \right) \\
& \subseteq \left( f_1(t_i) \times \cols_2 \right) \cup \left( \left( m_1 \times D_1 \right)  \times \{v\} \right ). 
\end{aligned}
$$
Instead, we will need to modify $d_m$ to eliminate the component of the support in $f_1(t_i) \times \cols_2$.  In the previous case we showed that the vector  $d_{(u,m_2)}$ has $\{(u,m_2)\} \subseteq \supp(d_{(u,m_2)}\cmat) \subseteq \{(u,m_2)\} \cup D$.  Thus, we can set $d'_m$ to be $d_m$ minus an appropriate linear combination of vectors in $Q = \{d_{(u,m_2)} | u \in f_1(t_i), m_2 \in \cols_2\}$ to obtain the desired support for $d_m'\cmat$.  Vectors in $Q$ have support in $f_1(t_i) \times \cols_2 = f(t_i)$, as needed.

\noindent  \textbf{$\mathbf{C}$ is $\mathbf{\rho}$-certifiable}


First, showing that Definition~\ref{def:netcod} goes through if $N(u)$ is replaced with clique $K(u)$ is identical to the proof above along with the observation that if $K_1$ and $K_2$ are cliques in $N_1$ and $N_2$ then $K_1 \times K_2$ is a clique in $N_1 \p N_2$. 

It remains to show that $\rk(\cmat^T I_M) \ge \rho$ for all multicuts $M$ of $N_1 \p N_2$.  

Notice that we can view the matrix $\cmat^T$ as having a block of rows for each $w \in V_1 \cup V_2$; the block of rows associated to $u \in V_1$ is $\ind{u} \krn \cmat^T_2$, and to $v \in V_2$ is $\cmat^T_1 \krn \ind{v}$ (where $\ind{u}$ is the indicator row vector of $u$).
  
We will show that $\rk(\cmat^TB) \ge \rho$ for a matrix $B$ that is in the column space of $I_M$.  This is sufficient because there is some linear transformation $T$ such that $I_MT = B$, implying $\rk(\cmat^TI_M) \ge \rk(\cmat^TI_MT) = \rk(\cmat^TB) \ge \rho$.

The matrix $B$ will have $r_1$ columns for each $v \in V_2$ and $r_2$ columns for each $u \in V_1$.  Let $M_u$, $u \in V_1$ be the multicut of $\{u\} \times V_2$ satisfying the conditions of Lemma \ref{lem:extracut} using the clique $K_1(u)$ that shows certifiability, and similarly for $M_v, v \in V_2$.  The matrix $B$ has a block of columns equal to $a_u^T \krn I_{M_u}$ for each $u \in V_1$, and $I_{M_v} \krn a^T_v$ for $v \in V_2 \setminus f_2(\src_2)$  where $a_u$ and $a_v$ are the vectors satisfying Definition~\ref{def:netcod} with cliques $K_1(u)$ and $K_2(v)$.  The matrix $B$ lies in the column space of $I_{M}$ because $a_u$ and $a_v$ have support within their corresponding cliques and $K_1(u) \times M_u \subseteq M$, $M_v \times K_2(v) \subseteq M$.    

We will show that the matrix $\cmat^TB$ is lower block triangular with $n_1$ diagonal blocks of rank at least $\rho_2$ and $n_2 - |f_2(\src_2)|$ diagonal blocks of rank at least $\rho_1$.  Row blocks of $\cmat^TB$ are indexed by $w \in V_1 \cup V_2$ and column blocks are indexed by $w \in V_1 \cup V_2 \setminus f_2(\src_2)$.  We will assume that the blocks are ordered according to $-\pi_1$ and $-\pi_2$ and blocks associated to elements of $V_1$ precede those of $V_2$.

We have four types of blocks in the product, we analyze all but the lower right block, which is irrelevant for purposes of showing the matrix is lower block triangular.

Block $[u,u']$, $u, u' \in V_1$: \hfill \\
$$\begin{aligned}
\cmat^TB[u,u'] & = (\ind{u} \krn \cmat^T_2) (a^T_{u'} \krn I_{M_{u'}}) \\
& = \ind{u}a_{u'}^T \krn \cmat^T_2I_{M_{u'}}\\
\end{aligned}$$
Thus, block $[u',u]$ has rank at least $\rho_2$ if $u \in \supp(a_{u'}) \subseteq K_1(u')$ and is identically zero otherwise.  In particular, it is zero whenever $\pi_1(u) > \pi_1(u')$ because $u \in K_1(u') \implies \pi_1(u) \le \pi_1(u')$.

Block $[v,v']$, $v, v' \in V_2\setminus f_2(\src_2)$: \hfill \\
$$\begin{aligned}
\cmat^TB[v,v'] & = (\cmat^T_1 \krn \ind{v}) (I_{M_{v'}} \krn a^T_{v'}) \\
& = \cmat^T_1I_{M_{v'}} \krn \ind{v}a_{v'}^T
\end{aligned}
$$
Just as for block $[u,u']$, block $[v,v']$ has rank at least $\rho_1$ if $v \in \supp(a_{v'})$ and is zero otherwise.

Block $[u,v]$, $u \in V_1, v \in V_2\setminus f_2(\src_2)$: \hfill \\
$$
\begin{aligned}
\cmat^TB[u,v] & = (\ind{u} \krn \cmat^T_2) (I_{M_{v}} \krn a_{v}^T)  \\
& = \ind{u}I_{M_{v}}\krn \cmat^T_2a_{v}^T = 0
\end{aligned}
$$

Where the last equality holds because $v \notin f_2(\src_2)$ implies $f_2^{-1}(v) = \emptyset$ and thus $\cols(f_2^{-1}(v)) = \emptyset$, giving $\cmat^T_2a_{v}=0$.

The first two cases above, along with the ordering of blocks so that larger $\pi$ values are on the top left, implies that the top left and lower right quadrants of the matrix $\cmat^TB$ are 
lower block triangular with the required ranks on the diagonal blocks.  
The final case implies that the top right quadrant is all zero, as wanted.
\end{proof}

\begin{proof}[Proof of Lemma \ref{lem:lowertri}]
Let $D_1, \ldots, D_l$ be the diagonal blocks of the matrix with ranks $r_1, \ldots, r_l$ respectively.  We can convert the matrix to the identity matrix starting with the top left diagonal block $D_1$.  First we apply steps of Gaussian elimination that convert $D_1$ to the identity of size $r_1$, possibly with additional rows or columns of all zeros.  We delete the zero rows and columns of $D_1$ from the entire matrix.  Then we subtract rows of $D_1 = I_{r_1}$ from the rest of the matrix so that the only non-zero terms in the first $r_1$ columns are contained in $D_1$.  Notice that the lower block triangular property implies that all of the preceding row operations only change the first $r_1$ columns.  We continue in this fashion for $D_2, \ldots, D_l$.  At the end we are left with an identity matrix of size $\sum_{i=1}^l r_i$, implying that our original matrix has a submatrix of rank at least $\sum_{i=1}^l r_i$.
\end{proof}

\begin{proof}[Proof of Lemma \ref{lem:extracut}]
Suppose for contradiction that there is a multicut $M$ of $N_1 \p N_2$ and some $u \in V_1$ such that for any multicut $M_2$ of $N_2$ there is at least one vertex $(a,b) \in K_1(u) \times M_2, (a,b) \notin M$.  Let $C = \{v \in V_2| K_1(u) \times  v \subseteq M\}$.  By assumption, $C$ is not a multicut of $N_2$, and there exists a source-sink path in $N_2$ that does not intersect with $C$.  Let $p_1\ldots p_l$ be such a path.  For each vertex $v \in V_2 \setminus C$, let $g(v) = (a,v)$ such that $a \in K_1(u), (a,v) \notin M$.  Such a vertex must exist by definition of $C$.  The path $g(p_1) \ldots g(p_l)$ is a source-sink path in $N_1 \p N_2$ that does not intersect $M$, a contradiction.     
\end{proof}


\section{Conclusions and Open Questions}
In this work we give a class of network codes that provide lower bounds on the multicut.  
There are many potential directions to expand this class.  For example, it may be possible to allow for edge-capacitated graphs or arbitrary capacities, or relax the condition of certifiability by strengthening Lemma \ref{lem:extracut}.  

In networks of Saks et al. we show the coding rate exactly matches the multicut, despite the flow being a factor $k$ smaller.  We know a simple example where the network coding rate is less than the multicut, 
but we have no example eliminating the possibility that just two times the network coding rate is always at least the multicut.  In general, does there exist some parameter $\beta$ that is $o(k)$ such that the coding rate scaled up by $\beta$ is always at least the size of the minimum multicut?  


\section{Acknowledgements}
I would like to thank Robert Kleinberg for many useful discussions and Jesse Simons for help editing.

\bibliographystyle{IEEEtranS}
\bibliography{productgraph}

\begin{thebibliography}{10}
\providecommand{\url}[1]{#1}
\csname url@samestyle\endcsname
\providecommand{\newblock}{\relax}
\providecommand{\bibinfo}[2]{#2}
\providecommand{\BIBentrySTDinterwordspacing}{\spaceskip=0pt\relax}
\providecommand{\BIBentryALTinterwordstretchfactor}{4}
\providecommand{\BIBentryALTinterwordspacing}{\spaceskip=\fontdimen2\font plus
\BIBentryALTinterwordstretchfactor\fontdimen3\font minus
  \fontdimen4\font\relax}
\providecommand{\BIBforeignlanguage}[2]{{%
\expandafter\ifx\csname l@#1\endcsname\relax
\typeout{** WARNING: IEEEtran.bst: No hyphenation pattern has been}%
\typeout{** loaded for the language `#1'. Using the pattern for}%
\typeout{** the default language instead.}%
\else
\language=\csname l@#1\endcsname
\fi
#2}}
\providecommand{\BIBdecl}{\relax}
\BIBdecl

\bibitem{dahlhaus1994complexity}
E.~Dahlhaus, D.~S. Johnson, C.~H. Papadimitriou, P.~D. Seymour, and
  M.~Yannakakis, ``The complexity of multiterminal cuts,'' \emph{SIAM Journal
  on Computing}, vol.~23, no.~4, pp. 864--894, 1994.

\bibitem{chuzhoy}
J.~Chuzhoy and S.~Khanna, ``Polynomial flow-cut gaps and hardness of directed
  cut problems,'' \emph{Journal of the ACM (JACM)}, vol.~56, no.~2, p.~6, 2009.

\bibitem{LeightonRao}
T.~Leighton and S.~Rao, ``Multicommodity max-flow min-cut theorems and their
  use in designing approximation algorithms,'' \emph{J. ACM}, vol.~46, no.~6,
  pp. 787--832, 1999.

\bibitem{garg-approx}
N.~Garg, V.~V. Vazirani, and M.~Yannakakis, ``Approximate max-flow min-(multi)
  cut theorems and their applications,'' \emph{SIAM Journal on Computing},
  vol.~25, no.~2, pp. 235--251, 1996.

\bibitem{cut-approx}
\BIBentryALTinterwordspacing
A.~Agarwal, N.~Alon, and M.~S. Charikar, ``Improved approximation for directed
  cut problems,'' in \emph{Proceedings of the thirty-ninth annual ACM symposium
  on Theory of computing}, ser. STOC '07.\hskip 1em plus 0.5em minus
  0.4em\relax New York, NY, USA: ACM, 2007, pp. 671--680. [Online]. Available:
  \url{http://doi.acm.org/10.1145/1250790.1250888}
\BIBentrySTDinterwordspacing

\bibitem{gupta2003improved}
A.~Gupta, ``Improved results for directed multicut,'' in \emph{Proceedings of
  the fourteenth annual ACM-SIAM symposium on Discrete algorithms}.\hskip 1em
  plus 0.5em minus 0.4em\relax Society for Industrial and Applied Mathematics,
  2003, pp. 454--455.

\bibitem{cheriyan2001approximating}
J.~Cheriyan, H.~Karloff, and Y.~Rabani, ``Approximating directed multicuts,''
  in \emph{Foundations of Computer Science, 2001. Proceedings. 42nd IEEE
  Symposium on}.\hskip 1em plus 0.5em minus 0.4em\relax IEEE, 2001, pp.
  320--328.

\bibitem{Saks}
\BIBentryALTinterwordspacing
M.~Saks, A.~Samorodnitskyâ€, and L.~Zosinâ€, ``\BIBforeignlanguage{English}{A
  lower bound on the integrality gap for minimum multicut in directed
  networks},'' \emph{\BIBforeignlanguage{English}{Combinatorica}}, vol.~24, pp.
  525--530, 2004. [Online]. Available:
  \url{http://dx.doi.org/10.1007/s00493-004-0031-x}
\BIBentrySTDinterwordspacing

\bibitem{li2003linear}
S.-Y. Li, R.~W. Yeung, and N.~Cai, ``Linear network coding,'' \emph{Information
  Theory, IEEE Transactions on}, vol.~49, no.~2, pp. 371--381, 2003.

\bibitem{Jaggi}
S.~Jaggi, P.~Sanders, P.~Chou, M.~Effros, S.~Egner, K.~Jain, and L.~Tolhuizen,
  ``Polynomial time algorithms for multicast network code construction,''
  \emph{Information Theory, IEEE Transactions on}, vol.~51, no.~6, pp. 1973 --
  1982, june 2005.

\bibitem{agarwal2004advantage}
A.~Agarwal and M.~Charikar, ``On the advantage of network coding for improving
  network throughput,'' in \emph{Information Theory Workshop, 2004.
  IEEE}.\hskip 1em plus 0.5em minus 0.4em\relax IEEE, 2004, pp. 247--249.

\bibitem{Harvey2006capacity}
N.~J.~A. Harvey, R.~Kleinberg, and A.~R. Lehman, ``On the capacity of
  information networks,'' \emph{IEEE Transactions on Information Theory},
  vol.~52, no.~6, pp. 2345--2364, 2006.

\bibitem{li2004network}
Z.~Li and B.~Li, ``Network coding: the case of multiple unicast sessions,'' in
  \emph{Allerton Conference on Communications}, 2004.

\bibitem{harvey2004comparing}
N.~J. Harvey, R.~D. Kleinberg, and A.~R. Lehman, ``Comparing network coding
  with multicommodity flow for the k-pairs communication problem,'' 2004.

\bibitem{ACLY}
R.~Ahlswede, N.~Cai, S.-Y.~R. Li, and R.~W. Yeung, ``Network information
  flow,'' \emph{IEEE Transactions on Information Theory}, vol.~46, no.~4, pp.
  1204--1216, 000.

\bibitem{harvey2005tighter}
N.~J. Harvey and R.~Kleinberg, ``Tighter cut-based bounds for k-pairs
  communication problems,'' in \emph{43rd Allerton Conference on Communication,
  Control, and Computing, Monticello, IL}, 2005.

\bibitem{KramerSavari}
G.~Kramer and S.~Savari, ``Edge-cut bounds on network coding rates,''
  \emph{Journal of Network and Systems Management}, vol.~14, no.~1, pp. 49--67,
  2006.

\bibitem{song2003zero}
L.~Song, R.~W. Yeung, and N.~Cai, ``Zero-error network coding for acyclic
  networks,'' \emph{Information Theory, IEEE Transactions on}, vol.~49, no.~12,
  pp. 3129--3139, 2003.

\bibitem{ARL-thesis}
A.~R. Lehman, ``Network coding,'' Ph.D. dissertation, MIT, 2005.

\end{thebibliography}
\end{document}